\documentclass{article}

\usepackage{amsmath}
\usepackage{amsfonts}
\usepackage{graphicx}

\reversemarginpar

	\usepackage[utf8]{inputenc}
	\usepackage{setspace}
	\usepackage{amsmath,amsthm, amssymb}
	\usepackage[normalem]{ulem} 

	\newcommand{\Z}{\mathbb{Z}}
	
	\newcommand{\PGL}{\rm{PGL}}
	\newcommand{\GL}{\rm{GL}}
	\newcommand{\pz}{\rm{PGL}(2,\Z)}

\newtheorem{theorem}{Theorem}

\newtheorem{proposition}{Proposition}

\theoremstyle{definition}

\begin{document}

\title{Finite Subgroups of the Extended Modular Group 
}
\author{Gregory { Dresden},
Prakriti {Panthi}, 
Anukriti {Shrestha}, Jiahao {Zhang},\\
Washington \& Lee University, Lexington, VA, USA,\\
corresponding author: {\tt dresdeng@wlu.edu} }
\date{} 

\maketitle

\begin{abstract}
We show that in the extended modular group 
$\overline\Gamma = \PGL(2, \mathbb{Z})$
there are exactly seven finite subgroups up to conjugacy;
three subgroups of size 2, one subgroup each of size 3, 4, and 6, and the
trivial subgroup of size 1.\\

\end{abstract}

\section{Introduction.}

In a recent article \cite{Bea}, Beauville used Galois cohomology to find all finite subgroups (up to conjugacy)
of $\PGL(2, K)$ for certain fields $K$. 
In this paper, we use elementary methods to do the same for 
the group 
$\overline\Gamma = \PGL(2, \Z)$, often called the
{\em extended modular group}. 
 Although this result can be derived from earlier work (both
 Klemm \cite[Satz 7.9]{Kle} and Newman \cite[Chapter IX, \S 14]{New} 
classify the finite  subgroups of 
 $\GL(2,\Z)$, from which we can obtain our result on $\PGL(2,\Z)$),
 we feel it deserves more exposure. Klemm's work was in the context of 
 classifying the wallpaper groups, and Newman used theory from linear algebra. Our paper, in contrast,
 calls upon two recent results (one on free groups using a theorem of Kurosh, the other on symmetry groups) to do the ``heavy lifting", and we 
 combine them to get our main result on $\PGL(2,\Z)$ using just basic algebra and direct computation. 

We define the group 
$\overline\Gamma = \PGL(2, \Z) = \rm{GL}(2, \Z)/{\{\pm I\}}$
as the set of all 
matrices 
$ \left( 
\begin{array}{cc} a & b \\ c & d \\
\end{array} \right)
$ 
with integer coefficients $a,b,c,d$ and $ad-bc = \pm1$ with the understanding
that 
$ \left( 
\begin{array}{cc} a & b \\ c & d \\
\end{array} \right)
=
 \left( 
\begin{array}{cc} -a & -b \\ -c & -d \\
\end{array} \right)
$.
One nice feature of this group $\overline\Gamma$ is that there is an isomorphism from
 $\overline\Gamma$ to a group of functions (called {\em linear fractional transforms}) under composition, as follows:
\[
 \left( 
\begin{array}{cc} a & b \\ c & d \\
\end{array} \right)
\ \  \longmapsto \ \ 
\frac{ ax+b}{cx+d}
\]
Thanks to this isomorphism, we can re-write the product of matrices 
as a composition of linear fractional transforms, and vice-versa. 

In what follows, we will use
matrix notation and function notation interchangeably. For example, we will
use 
the matrix 
$ \left( 
\begin{array}{cc} 0 & 1 \\ 1 & 0 \\
\end{array} \right)
$,
the matrix 
$ \left( 
\begin{array}{cc} 0 & -1 \\ -1 & 0 \\
\end{array} \right)
$,
and the corresponding function $\displaystyle \frac{1}{x}$ to refer to the same object in $\overline\Gamma = \PGL(2, \Z)$
as convenient. Because of this correspondence, we can define  
the ``determinant" of the function $(ax+b)/(cx+d)$ to be
the determinant $ad-bc$ of the corresponding matrix 
$ \left( 
\begin{array}{cc} a & b \\ c & d \\
\end{array} \right)
$.
For more on the extended modular group $\overline\Gamma$, 
see for example
\cite{JT, Kul, SIK, Sin}.

\section{Statement of Main Result.} 

To find all 
finite subgroups of $\overline\Gamma = \PGL(2, \mathbb{Z})$ up to conjugacy,
we need to carefully stitch together two previous results from 2003 and 2004.

This first theorem, by Y{\i}lmaz \"Ozg\"ur and {\c S}ahin \cite[Theorem 2.3]{OS}, comes from considering the  presentation of
$\overline\Gamma$ as a free group with three generators, and it gives us
{\em elements} in $\overline\Gamma$ of finite order, up to {\em conjugacy}. 

\begin{theorem}[Y{\i}lmaz \"Ozg\"ur, {\c S}ahin] \label{ozg} 
There are exactly four conjugacy classes for non-trivial elements of finite order in
$\overline\Gamma$. Every element of order two is 
conjugate to either $1/x$ or $-x$ or $-1/x$, 
and  every element of order three is conjugate to 
$-1/(x+1)$.
\end{theorem}
This second result, derived from a paper by Dresden \cite{Dre1}, 
comes from considering the finite symmetry groups of the sphere, and it gives us
{\em subgroups} in $\overline\Gamma$ of finite order, up to {\em isomorphism}. 

\begin{theorem}\label{dre} 
There are exactly four  isomorphism classes for non-trivial subgroups of finite 
order in
$\overline\Gamma$. Every such subgroup is isomorphic to either one of the
 cyclic groups $C_2, C_3$, or one of the dihedral
groups $D_2, D_3$, of 
sizes 2, 3, 4, and 6 respectively. \end{theorem}
(We will prove this theorem in a moment.) 
We will be able to combine these two theorems to prove our main result, which 
we state here.

\begin{theorem}\label{main}
Any finite non-trivial subgroup of $\overline\Gamma = \PGL(2, \mathbb{Z})$ is of size
two, three, four, or six. 
The groups of size two are conjugate in $\overline\Gamma$ 
to either $\{ x, -x \}$ or 
$\{ x, 1/x \}$  or
$\{ x, -1/x \}$.
All groups of size three in $\overline\Gamma = \PGL(2, \mathbb{Z})$ 
are conjugate 
in $\overline\Gamma$ to 
\[
G_3 = \left\{ x, \ \frac{-1}{x+1},\  \frac{-x-1}{x} \right\}.
\]
Likewise, all groups of size four are conjugate to 
\[
G_4 = \left\{ x,\  \frac{1}{x}, \  -x, \  \frac{-1}{x} \right\}.
\]
and all groups of size six are conjugate to
\[
G_6 = \left\{ x, \ \frac{-1}{x+1},\  \frac{-x-1}{x},
\ \  \frac{1}{x},\  \frac{-x}{x+1}, \ 
-x-1 \right\}.
\]
\end{theorem}

\section{Proofs.}

\begin{proof}[Proof of Theorem \ref{dre}]
Thanks to the isomorphism mentioned earlier, 
we can think of  $\overline\Gamma = \PGL(2, \mathbb{Z})$ 
as the group of 
linear fractional transforms $(ax+b)/(cx+d)$ with integer coefficients
and determinant $ad - bc = \pm1$. 
This group $\overline\Gamma$ sits inside the larger group 
of such linear fractional transforms with {\em non-zero} determinant
$ad-bc$, and 
we call upon Theorem 1 of \cite{Dre1}
to see that all non-trivial finite subgroups
of this larger group (and hence of our group $\pz$) 
are isomorphic to either $C_2$, $C_3$, $C_4$, $C_6$,
$D_2$, $D_3$, $D_4$, or $D_6$. By Theorem \ref{ozg} of this paper we see that
$\PGL(2, \Z)$ does not have elements of order 4 or 6, thus eliminating
from consideration the groups $C_4$, $C_6$, $D_4$, and $D_6$. 
It remains to show that the other finite groups are realizable in 
$\pz$, but this follows from the explicit examples given in the statement of our
Theorem \ref{main}.
%
\end{proof}

The following proposition is essential to our proof of Theorem \ref{main}, 
and will allow us to 
combine together our Theorems \ref{ozg} and \ref{dre}, above. 
\begin{proposition}\label{single4}
If a subgroup of  $\overline\Gamma = \PGL(2, \mathbb{Z})$ 
contains $-1/x$ and is of size 4, then it must equal $G_4$.
Likewise, if a subgroup of  $\overline\Gamma = \PGL(2, \mathbb{Z})$ 
contains $-1/(x+1)$ and is of size 6, then it must equal $G_6$.
\end{proposition}

\begin{proof} 
We begin by noting that any element in 
$\pz$ of order 2 must have matrix form
$ \left( 
\begin{array}{cc} a & b \\ c & -a \\
\end{array} \right)
$.
This is easy to see if we note that 
\[
\left( 
\begin{array}{cc} a & b \\ c & d \\
\end{array} \right)\cdot
\left( 
\begin{array}{cc} a & b \\ c & d \\
\end{array} \right)   =    
\left( 
\begin{array}{cc} a^2 + bc & b(a+d) \\ c(a+d) & d^2+bc \\
\end{array} \right),
\]
and for this to equal the identity in $\pz$
either we have $d = -a$ as desired, or we have $b=c=0$ (which
forces $a^2$ and $d^2$ to be 1, and combined with our
matrix being of order 2 and not order 1, this forces
$d = -a$ as desired). We also recall that our definition of $\pz$ 
requires that $a^2 + bc = \pm 1$. 

We turn now to subgroups of $\pz$ of size 4. 
Let  $F_4$ be such a subgroup, and suppose 
$F_4$ contains $-1/x$. We know 
from Theorem \ref{dre} that $F_4$ is dihedral, 
so it contains another element (call it $p(x)$) also of 
order 2. From our discussion above, we 
can write $p(x)$ in function form as $p(x) = (ax+b)/(cx-a)$ with 
$a^2+bc = \pm 1$, and since
a dihedral group of size 4 is abelian then 
$p(-1/x) = -1/p(x)$. In matrix form, this becomes 
$ \left( 
\begin{array}{cc} 
	b & -a \\ 
	-a & -c 
\end{array} \right) = 
 \pm \left( 
\begin{array}{cc} 
	c & -a \\ 
	-a & -b 
\end{array} \right)$,
and by examining the various cases (and recalling that $a^2 + bc = \pm 1$)
we conclude that either $p(x) = -x$ or $p(x) = 1/x$, thus giving us $F_4$
equal to our group $G_4$. 

Finally, we consider 
a subgroup (call it $F_6$) of size 6 in $\pz$ which contains
$m(x) = -1/(x+1)$ and thus also $m^2(x) = (-x-1)/x$.
From Theorem \ref{dre} we know $F_6$ is dihedral, 
so it contains another element (call it $p(x)$) also of 
order 2 such that 
$m^2(p(x)) = p(m(x))$. As seen earlier, we 
can write $p(x) = (ax+b)/(cx-a)$ with 
$a^2+bc = \pm 1$, and
our equality 
$m^2(p(x)) = p(m(x))$
in matrix form becomes
$ \left( 
\begin{array}{cc} 
	-a-c & a-b \\ 
	a & b 
\end{array} \right) = 
 \pm \left( 
\begin{array}{cc} 
	b & b-a \\ 
	-a & -a-c 
\end{array} \right)$.
If we first consider the ``$+$" in the ``$\pm$" above, we quickly arrive at
$a=0$ and thus $b=c=0$, a contradiction. If we now
consider the  ``$-$" in the ``$\pm$" above, we get $b=a+c$, and substituting
this into $a^2 + bc = \pm 1$ gives us $a^2 + ac + c^2 = \pm 1$. By looking at the 
possible values of $a$ and $c$ (namely, $-1, 0$, and $1$) we arrive at
$p(x) = -x-1$ or $p(x) = 1/x$ or $p(x) = -1/(x+1)$, and so our 
group $F_6$ equals $G_6$ as desired.
\end{proof}

It is now an easy matter to prove our main result.

\begin{proof}[Proof of Theorem \ref{main}]
The case for groups of size two and three follows immediately from Theorem \ref{ozg}. 
If $G$ is a group of size 6, then by Theorem \ref{dre}
it is dihedral with an element of order three; by Theorem \ref{ozg} we
can conjugate it
to get a new group $G'$ containing
$-1/(x+1)$, and by Proposition \ref{single4} this new group $G'$ must equal $G_6$. 

The remaining case where 
 $G$ is 
a group 
of size four is a bit more challenging.  
By Theorem \ref{dre} we know $G$ is dihedral with three elements of order two;
we can thus write $G = \{I, A, B, AB\}$ for $A, B$, and $AB$ matrices with
determinants $\pm 1$. At least one 
of these three matrices in $G$ must have determinant
$1$. Now, determinants are preserved under conjugacy in $\pz$, and 
by Theorem \ref{ozg} anything of order 2 in $\pz$ must be 
conjugate
to either $1/x$ (with associated determinant $-1$)
or $-x$ (with associated determinant $-1$) or $-1/x$ (with 
associated determinant $1$). 
Thus, our group $G$ must have an element conjugate to $-1/x$, and 
so we can conjugate our group $G$ to get a new group $G'$ containing
$-1/x$ and then apply Proposition \ref{single4} to state that this new group $G'$ must equal $G_4$. 
\end{proof}

\end{document}